\documentclass[12pt]{amsart}
\usepackage{amssymb,latexsym}
\usepackage{enumerate}
\usepackage[margin=2in]{geometry}
\usepackage{graphicx}
\usepackage{psfrag}
\makeatletter
\@namedef{subjclassname@2010}{%
  \textup{2010} Mathematics Subject Classification}
\makeatother

\newtheorem{thm}{Theorem}

\newtheorem{prop}{Proposition}

\theoremstyle{definition}

\newtheorem*{xrem}{Remark}

\numberwithin{equation}{section}

\newcommand{\srpot}[1]{\mathfrak{N}_{#1}}

\begin{document}


\baselineskip=17pt



\title{On negative results concerning Hardy means}

\author{Pawe{\l} Pasteczka}
\address{Institute of Mathematics\\University of Warsaw\\
02-097 Warszawa, Banach str. 2, Poland}
\email{ppasteczka@mimuw.edu.pl}

\date {November 09, 2013}

\begin{abstract}
In the present paper we are going to prove some necessary condition for a mean to be Hardy. This condition is then applied to completely characterize the Hardy property among the Gini means.
\end{abstract}



\maketitle

\section{Introduction}


A function $\mathfrak{A} \colon \bigcup_{n=1}^{\infty} I^n \rightarrow I$ is a 
\emph{Mean} ($I$ -- an interval) if 
\begin{itemize}
\item[(i)]$\mathfrak{A}$ is non-decreasing with respect to each variable, 
\item[(ii)]$\mathfrak{A}(\underbrace{a,\ldots,a}_{n})=a$ for any $a \in I$ and $n \in \mathbb{N}$. 
\end{itemize}

In 2004 P\'ales and Persson, \cite{PaPe04}, proposed the following definition for a mean to be \emph{Hardy}:
Let $I \subset \mathbb{R}_{+}$ be an interval, $\inf I=0$. A mean $\mathfrak{A}$ defined on an interval $I$ is \emph{Hardy} if there exists a constant $C$ such that for any $a \in l^1(I)$ 
$$\sum_{n=1}^{\infty} \mathfrak{A}(a_1,\ldots,a_n)<C\sum_{n=1}^{\infty} a_n.$$

In fact, such a definition had been felt in the air since the year 1920 when it was proved that the power mean
$$\srpot{\lambda}(a_1,\ldots,a_n):=
\begin{cases} 
\min(a_1,\ldots,a_n)	& \textrm{if\ } \lambda = - \infty\,,\\
\Big(\frac{1}{n} \sum_{k=1}^{n} a_k^\lambda\Big)^{1/\lambda}	& \textrm{if\ } \lambda \ne 0\,,\\
\Big(\prod_{k=1}^{n} a_k\Big)^{1/n}	& \textrm{if\ } \lambda = 0\,,\\
\max(a_1,\ldots,a_n)	& \textrm{if\ } \lambda = + \infty
\end{cases}$$
 ($I=\mathbb{R}_{+}$) was Hardy if and only if $\lambda<1$ (cf. \cite{Hardy20}). 

During the last hundred years many results where obtained in the field of Hardy Means. The reader may find them in catching surveys \cite{PS,DMcG,OP}, and in a recent book \cite{KMP}. In the present paper we are going to give a necessary condition for a mean to be Hardy. Namely we are going to prove the following 

\begin{thm}
\label{thm:main}
Let $\mathfrak{A}$ be a mean defined on an interval $I$, $(a_n)$ be a sequence of positive numbers in $I$ satisfying $\sum a_n = + \infty$ and \,$\lim a_n=0$.

If\, $\lim a_n^{-1} \mathfrak{A}(a_1,\ldots a_n) = \infty$ then $\mathfrak{A}$ is not Hardy.
\end{thm}

\begin{proof}
Let us suppose conversely that $\mathfrak{A}$ is a Hardy mean with a constant $C>0$.
Let us fix $n_0$ and $n_1>n_0$ such that 
\begin{align}
a_n^{-1} \mathfrak{A}(a_1,\ldots a_n) &> 2\,C \textrm{ for any }n>n_0, \nonumber \\
\sum_{i=n_0+1}^{n_1-1} a_n &>\sum_{i=1}^{n_0} a_n. \nonumber
\end{align}

Let $b_n=\begin{cases} a_n &\textrm{, for }n\le n_1, \\ a_{n_1}2^{-n} &\textrm{, for }n> n_1\end{cases}.$
The sequence $(b_n) \in l^1(I)$ will give a contradiction. Indeed, 
\begin{align}
\sum_{n=1}^{\infty} \mathfrak{A}(b_1,\ldots,b_n) &> \sum_{n=n_0+1}^{n_1} \mathfrak{A}(b_1,\ldots,b_n) \nonumber \\
&\ge \sum_{n=n_0+1}^{n_1} 2\,C a_n \nonumber \\
&= C \left( \sum_{n=n_0+1}^{n_1-1} a_n + a_{n_1} + \sum_{n=n_0+1}^{n_1} a_n \right) \nonumber \\
&> C \left( \sum_{n=1}^{n_0} a_n +  \sum_{n=n_1+1}^{\infty} b_n + \sum_{n=n_0+1}^{n_1} a_n \right) \nonumber \\
&= C \sum_{n=1}^{\infty} b_n. \nonumber
\end{align}

\end{proof}

\section{Applications}

In a moment we are going to present applications of Theorem~\ref{thm:main} to two fairly famous families of means.
In both cases necessary and sufficient conditions will be presented. The relevant proofs will be postponed till the next section.

\subsection{Gaussian Product}

Power means were generalized in different ways by many authors (cf., e.g., \cite[chap. III-VI]{bullen} for details).
In particular, in 1947, Gustin \cite{Gustin47} proposed an extension of an earlier Gauss' concept \cite[pp. 361--403]{Gauss}. 

Namely, let $\lambda=(\lambda_1,\ldots,\lambda_p) \in \mathbb{R}^n$ and $v$ be an all-positive-components vector. One defines a sequence
\begin{align}
v^{(0)}&=v, \nonumber \\
v^{(i+1)}&=\left(\srpot{\lambda_1}(v^{(i)}),\srpot{\lambda_2}(v^{(i)}),\ldots,\srpot{\lambda_p}(v^{(i)})\right). \nonumber 
\end{align}

Then it is known that the limit $\lim_{i \rightarrow \infty} \srpot{\lambda_k}(v^{(i)})$ exists and does not depend on $k$.
This common limit is denoted by $\srpot{\lambda_1} \otimes \cdots \otimes \srpot{\lambda_p}$.
Because of various Gauss' results on $\srpot{1} \otimes \srpot{0}$, such means are called Gaussian Means (or, more descriptively, the Gaussian product of Power Means).

We are going to give a necessary and sufficient condition for Gaussian Means to be Hardy. More precisely we are going to prove
\begin{thm}
\label{thm:gaussian}
Let $p \in \mathbb{N}$ and $\lambda \in \mathbb{R}^p$. 
,Then $\srpot{\lambda_0}\otimes \srpot{\lambda_1} \otimes \cdots \otimes \srpot{\lambda_p}$ is Hardy if and only if\,$\max(\lambda_0,\ldots,\lambda_p)<1$.
\end{thm}

\subsection{Gini Means}

Let $\mathfrak{G}_{p,q}(a_1,\ldots,a_n):=\left(\frac{\sum_{i=1}^n a_i^p}{\sum_{i=1}^n a_i^q}\right)^{1/(p-q)}$(cf. \cite{Gini} and \cite[p. 248]{bullen}).
In 2004 P\'ales and Persson proved the following
\begin{prop}[\cite{PaPe04},Theorem 2]
\label{prop:PaPe04Thm2}
Let $p,q \in \mathbb{R}$. If $\mathfrak{G}_{p,q}$ is a Hardy mean, then 
$$\min(p,q) \le 0 \textrm{ and } \max(p,q) \le 1.$$
Conversely, if
$$\min(p,q) \le 0 \textrm{ and } \max(p,q) < 1.$$
then $\mathfrak{G}_{p,q}$ is a Hardy mean.
\end{prop}
They also put forward a conjecture \cite[Open Problem 3.]{PaPe04} that the sufficient condition in proposition above is also a necessary one. We will justify this conjecture. Namely, we will prove the following
\begin{thm}
\label{thm:gini}
Let $p,q \in \mathbb{R}$. Then $\mathfrak{G}_{p,q}$ is a Hardy mean if and only if $\min(p,q)\le0$ and $\max(p,q)<1$ .
\end{thm}

\section{Proofs of Theorem~\ref{thm:gaussian} and Theorem~\ref{thm:gini}}

In both proofs used will the elementary estimations
\begin{align}
\sum_{i=1}^{n} i^k &\le n^{k+1} \quad \textrm{ for any } n \in \mathbb{N}, \label{eq:nk}\\
\sum_{i=1}^{n} \tfrac{1}{i} &\ge \ln n \quad \textrm{ for any } n \in \mathbb{N}. \label{eq:n-1}
\end{align}

\subsection{Proof of Theorem~\ref{thm:gaussian}}
Before we begin the proof, let us note that if $\lambda_i\le \lambda_i'$ for every $i=0,1,2,\ldots,p$, then
\begin{equation}
\srpot{\lambda_0}\otimes \srpot{\lambda_1} \otimes \cdots \otimes \srpot{\lambda_p} \le \srpot{\lambda_0'}\otimes \srpot{\lambda_1'} \otimes \cdots \otimes \srpot{\lambda_p'}. \label{eq:majotimes}
\end{equation}

Therefore, the $(\Leftarrow)$ part is simply implied by the fact that 
the mean $\srpot{\lambda_0}\otimes \srpot{\lambda_1} \otimes \cdots \otimes \srpot{\lambda_p}$ is
majorized by a Hardy mean $\srpot{\max(\lambda_1,\ldots,\lambda_p)}$, so it is Hardy too (recall that
$\max(\lambda_1,\ldots,\lambda_p)<1$).

Now we are going to prove the $(\Rightarrow)$ implication. One may assume that $i \mapsto \lambda_i$ is non-increasing. 
 Moreover, by \eqref{eq:majotimes}, having $\lambda_0\ge 1$ we estimate from below: $\lambda_0$ by $1$ and $\lambda_1,\,\lambda_2,\ldots,\,\lambda_p$ by $-\lambda$ for certain $\lambda>0$ and we are going to prove that
$$\mathfrak{A}:=\srpot{1}\otimes \underbrace{\srpot{-\lambda} \otimes \cdots \otimes \srpot{-\lambda}}_p$$
is not Hardy.

To this end, let us consider a two variable function $F(a,b):=\mathfrak{A}(a,\underbrace{b,\ldots,b}_{p})$ and fix $\theta>1$.
Then, for $a>\theta b$,
\begin{align}
F(a,b)&=
\mathfrak{A}\Bigg(\frac{a+pb}{p+1},\underbrace{\left( \frac{p+1}{a^{-\lambda}+pb^{-\lambda}}\right)^{1/\lambda},\ldots,\left( \frac{p+1}{a^{-\lambda}+pb^{-\lambda}}\right)^{1/\lambda}}_{p} \Bigg)\nonumber \\
&\ge \mathfrak{A}\Bigg(\tfrac{1}{p+1}a,\underbrace{\left( \frac{p+1}{(\theta b)^{-\lambda}+pb^{-\lambda}}\right)^{1/\lambda},\ldots,\left( \frac{p+1}{(\theta b)^{-\lambda}+pb^{-\lambda}}\right)^{1/\lambda}}_{p} \Bigg)\nonumber \\
&= \mathfrak{A}\Bigg(\tfrac{1}{p+1}a,\underbrace{\left( \frac{p+1}{\theta^{-\lambda}+p}\right)^{1/\lambda}b,\ldots,\left( \frac{p+1}{\theta^{-\lambda}+p}\right)^{1/\lambda}b}_{p}\Bigg) \nonumber \\
&= F\Bigg(\tfrac{1}{p+1}a,\left( \frac{p+1}{\theta^{-\lambda}+p}\right)^{1/\lambda}b\Bigg) \label{eq:FpropB}
\end{align}

Introduce two more mappings 
\begin{align}
\tau \colon (a,b) &\mapsto \left(\tfrac{1}{p+1}a,\left( \frac{p+1}{\theta^{-\lambda}+p}\right)^{1/\lambda}b\right), \nonumber \\
G \colon (a,b) &\mapsto \left( a^{\log_{p+1} \left( \tfrac{p+1}{\theta^{-\lambda}+p} \right)} b ^\lambda \right) ^{1/\left(\lambda+\log_{p+1} \left( \tfrac{p+1}{\theta^{-\lambda}+p} \right) \right)}. \nonumber 
\end{align}

With this notations there clearly hold
\begin{itemize}
\item $G(a,b) \in (\min(a,b),\max(a,b))$,
\item $F(a,b) \in (\min(a,b),\max(a,b))$,
\item $ G \circ \tau(a,b) =G(a,b)$,
\item $F$, $G$ and $\tau$ are homogeneous.
\end{itemize}
Moreover, inequality \eqref{eq:FpropB} assumes now a compact form
$$ F \circ \tau(a,b) \le F(a,b) \textrm{, for } a>\theta b.$$
We will prove that
\begin{equation}
F(a,b) > \tfrac{1}{\theta(p+1)}G(a,b) \textrm{ for any }a>b. \label{eq:F>G}
\end{equation}
The case when $\tfrac{a}{b}<\theta(p+1)$ is simply implied by first and second property.

Otherwise, let $a_0=a$, $a_0=b$, $(a_{i+1},b_{i+1})=\tau(a_i,b_i)$. By the definition of $\tau$, $a_n \rightarrow 0$ and $b_n \rightarrow +\infty$. Denote by $N$ the smallest natural number such that $a_N\le \theta b_N$.
Obviously $a_{N-1}>\theta b_{N-1}$, thus
\begin{align}
a_N = \tfrac{1}{p+1} a_{N-1}&>\tfrac{\theta}{p+1} b_{N-1}=\tfrac{\theta}{p+1} \left( \frac{\theta^{-\lambda}+p}{p+1} \right)^{1/\lambda} b_N \nonumber \\
&>\tfrac{\theta}{p+1} \left( \frac{\theta^{-\lambda}+\theta^{-\lambda}p}{p+1} \right)^{1/\lambda} b_N = \tfrac{1}{p+1} b_N. \nonumber
\end{align}
Hence
\begin{align}
F(a,b)&=F(a_0,b_0) \ge F\circ \tau^N (a_0,b_0) 
 = F(a_N,b_N) 
 \ge \min(a_N,b_N) \nonumber \\
&> \tfrac{1}{\theta}a_N 
 \ge \tfrac{1}{\theta (p+1)} \max(a_N,b_N) 
 \ge \tfrac{1}{\theta (p+1)} G(a_N,b_N) \nonumber \\
&=\tfrac{1}{\theta (p+1)} G \circ \tau^N(a_0,b_0)
 =\tfrac{1}{\theta (p+1)} G(a_0,b_0)=\tfrac{1}{\theta (p+1)} G(a,b). \nonumber
\end{align}

Lastly, using: -the fact that the mean $\mathfrak{A}$ is greater then its minimal argument, -homogeneity of $F$, -inequality \eqref{eq:n-1}, -inequality \eqref{eq:F>G}, one obtains
\begin{align}
(\tfrac1n)^{-1} \mathfrak{A}(1,\tfrac12,\tfrac13,\ldots,\tfrac1n) 
&= n F(\srpot{1}(1,\tfrac12,\tfrac13,\ldots,\tfrac1n),\srpot{-\lambda}(1,\tfrac12,\tfrac13,\ldots,\tfrac1n)) \nonumber \\
&\ge n F\left(\frac{\ln n}{n},\frac1n\right)=F(\ln n,1) 
 \ge \tfrac1{\theta(p+1)} G(\ln n,1) \nonumber \\
&\ge  \tfrac1{\theta(p+1)} (\ln n)^{\frac{\log_{p+1} \left( \tfrac{p+1}{\theta ^{-\lambda}+p} \right) }{\lambda+\log_{p+1} \left( \tfrac{p+1}{\theta^{-\lambda}+p} \right)}}. \nonumber
\end{align}

But, for any $\theta >1$ and $\lambda>0$, the right-most term tends to infinity when $n \rightarrow \infty$. So, by Theorem~\ref{thm:main}, $\mathfrak{A}$ is not Hardy.

\begin{xrem} Often the right-most term tends to infinity very slowly. For example ($\lambda=5$ and $p=3$) one obtains ($\theta=\tfrac{3}{2}$)
$$(\tfrac1n)^{-1} \srpot{1}\otimes \srpot{-5}\otimes \srpot{-5}\otimes \srpot{-5}(1,\tfrac12,\tfrac13,\ldots,\tfrac1n) > \tfrac16 (\ln n)^{0.0341}.$$
In particular it implies that the left hand side is greater than $1$ (a trivial estimation) for $n>10^{2.86\cdot 10^{22}}$.
\end{xrem}

\subsection{Proof of Theorem~\ref{thm:gini}}
The $(\Leftarrow)$ implication is implied by Proposition~\ref{prop:PaPe04Thm2}.

Working towards $(\Rightarrow)$ implication, let us assume that $\max(p,q)\ge 1$. We will prove that $\mathfrak{G}_{p,q}$ in not Hardy.
We know that if $p\le p'$ and $q\le q'$ then $\mathfrak{G}_{p,q} \le \mathfrak{G}_{p',q'}$ (cf. \cite[pp.249--250]{bullen}).
Moreover, $\mathfrak{G}_{p,q}=\mathfrak{G}_{q,p}$, hence 
\begin{equation}
\mathfrak{G}_{p,q} \ge \mathfrak{G}_{1,-k}\textrm{ for some }k \in \mathbb{N}. \label{eq:majorGini}
\end{equation}
From now on we assume that $p=1$ and $q=-k$ for some $k \in \mathbb{N}$.

Upon taking the sequence $a_i=\tfrac{1}{i}$, by \eqref{eq:nk} and \eqref{eq:n-1}, one obtains

$$a_n^{-1}\mathfrak{G}_{1,-k}(a_1,\ldots,a_n)
=n \left(\frac{\sum_{i=1}^n \tfrac{1}{i}}{\sum_{i=1}^n i^k}\right)^{1/(k+1)} 
\ge n\left(\frac{\ln n}{n^{k+1}}\right)^{1/(k+1)}=(\ln n)^{1/(k+1)}.$$
But $(\ln n)^{1/(k+1)} \rightarrow +\infty$  so, by Theorem~\ref{thm:main}, the Gini mean $\mathfrak{G}_{1,-k}$ is Hardy for no $k \in \mathbb{N}$. Hence, by~\eqref{eq:majorGini}, $\max(p,q) \ge1$ implies $\mathfrak{G}_{p,q}$ not being Hardy.


\begin{thebibliography}{99}

\bibitem{bullen} P.~S.~Bullen,
{\it Handbook of Means and Their Inequalities},
{Mathematics and Its Applications,} {vol. 560,} 
{Kluwer Acad. Publ.,} {Dordrecht 2003}.

\bibitem{DMcG} J.~Ducan, C.~M.~McGregor,
{\it Carleman's Inequality},
{Amer. Math. Monthly}
{\bf 110} (2003), 424--431.

\bibitem{Gini} C.~Gini,
{\it Di una formula compressiva delle medie},
Metron 
{\bf 13} (1938), 3--22.

\bibitem{Gustin47} W.~Gustin,
{\it Gaussian Means},
{Amer. Math. Monthly} {\bf 54} (1947), 332--335.

\bibitem{Gauss} C.~F.~Gauss,
{\it Werke} {\bf 3}, G\"ottingen-Leipzig, 1866.

\bibitem{Hardy20} G.~H.~Hardy,
{\it Note on a theorem of Hilbert},
{Math. Zeitschrift} {\bf 6} (1920), 314--317.

\bibitem{KMP} A.~Kufner, L.~Maligranda, L.-E.~Persson,
{\it The Hardy Inequality: About its History and Some Related Results},
{Vydavatelsk\v{y} Servis,} Pilsen 2007.


\bibitem{OP}  J.~A.~Oguntuase, L-E. Persson,
{\it Hardy type inequalities via convexity -- the journey so far},
{Aust. J. Math. Anal. Appl.} {\bf 7} (2011), 1--19.
 
\bibitem{PaPe04} Zs.~P\'ales, L.-E.~Persson,
{\it Hardy-type inequalities for means},
{Bull. Austral. Math. Soc.} {\bf 70} (2004), 521--528.

\bibitem{PS} J.~Pe\v{c}ari\`c,  K.~B.~Stolarsky,
{\it Carleman's inequality: history and new generalizations},
{Aequationes Math.} {\bf 61} (2001), 49--62.

\end{thebibliography}
\end{document}